\documentclass[final,4p,4pt, times]{article}
\usepackage{authblk}
\usepackage{amsmath, amsthm, mathrsfs, dsfont}
\usepackage{mathrsfs}
\usepackage{amssymb,latexsym}
\usepackage[german,english]{babel}
\usepackage{url}
\usepackage{graphicx}
\usepackage{gastex}
\usepackage{longtable}
\usepackage{lscape}
\usepackage{tabularx}
\usepackage{multicol}
\usepackage{verbatim}
\usepackage{multirow}
\usepackage{graphicx}
\usepackage{blkarray}   
\usepackage{epsfig,graphics,graphicx}
\usepackage{geometry}
\usepackage{float}
\usepackage{setspace}
\usepackage[hidelinks]{hyperref}
\usepackage{orcidlink}
\usepackage{tikz}

\hbadness=\maxdimen

\newtheorem{rem}{{\bf Remark}}
\newtheorem{deff}{Definition}[section]
\newtheorem{lem}{{\bf Lemma}}[section]
\newtheorem{thm}{{\bf Theorem}}[section]
\newtheorem{prop}{{\bf Proposition}}[section]
\newtheorem{cor}{{\bf Corollary}}[section]

\usepackage{chngcntr}
\counterwithout{equation}{section}
\newcounter{case}
\renewcommand{\thecase}{\alph{case}}

\begin{document}
\title{Characterisations of finite groups with exponent $q$ via their power graphs} 
\author{Aditya Singh$^1$\orcidlink{0009-0000-8323-4563} \  Anmol chugh$^1$\orcidlink{0009-0001-4233-1065} \ Yogendra Singh$^2$\orcidlink{0000-0001-6305-7168} \
Anand Kumar Tiwari$^1$\orcidlink{0000-0002-7038-1801}}

\date{
$^{1}$\small Department of Applied Sciences, Indian Institute of Information Technology, Allahabad 211015, India. \\ $^{2}$\small Department of Mathematics, Faculty of Sciences, Adani University, Ahmedabad 382421, India}

\maketitle
	
\hrule
\begin{abstract} The power graph $P(G)$ of a finite group $G$ is the graph with vertex set $G$ and edge set $E(P(G))=\{uv:\ u,v \in G,\ u \neq v,\ u \in \langle v \rangle \ \text{or}\ v \in \langle u \rangle\},$ where $\langle x\rangle$ denotes the cyclic subgroup generated by $x$. In this paper, we characterise all the finite groups with exponent $q$ whose power graphs are friendship graphs, firefly-type graphs, or torch graphs. We prove that the power graph of a finite group $G$ with exponent $q$ is a friendship graph if and only if $q=3$. In particular, in the abelian case, this is equivalent to $G\cong\mathbb{Z}_3^{n}$. We further show that, among all the symmetric and alternating groups, only $S_3$ and $A_4$ have firefly-type power graphs, whereas no finite group has a power graph isomorphic to a torch graph. Finally, we determine the generalised distance spectra $D_{\alpha}$-spectra of these graph classes.
 
\end{abstract}


\hrule

\section{Introduction}
Graphs arising from algebraic structures such as groups, rings, and vector spaces have attracted considerable attention in recent years, as they provide a natural connection between algebra and graph theory. Among these, the power graph provides an important graphical representation of semigroups and groups. The concept of the directed power graph of a semigroup was introduced by Kelarev and Quinn \cite{kq(2002)}. For a semigroup $S$, the directed power graph, denoted by $\overrightarrow{P}(S)$, has vertex set $S$, and there is a directed edge from $u$ to $v$ whenever $u\neq v$ and $v=u^{m}$ for some positive integer $m$. Subsequently, Chakrabarty et al. \cite{cgs(2009)} introduced the undirected power graph of a group. For a finite group $G$, the power graph $P(G)$ is the simple graph with vertex set $G$, in which two distinct vertices are adjacent precisely when one belongs to the cyclic subgroup generated by the other.

The foundational work of Chakrabarty et al. \cite{cgs(2009)} established that $P(G)$ is complete if and only if $G$ is a cyclic group of order $1$ or $p^{m}$, where $p$ is a prime and $m\in\mathbb{N} \cup \{0\}$. This result initiated a systematic study of the structural properties of power graphs. Later, Doostabadi et al. \cite{def(2014)} classified all finite groups whose power graphs are claw-free, $K_{1,4}$-free, or $C_{4}$-free. The study of structural properties of power graphs has continued in several directions. Ahmad et al. \cite{aa(2020)} characterised finite groups whose power graphs are unicyclic, bicyclic, or cactus graphs, and also determined the spectra of these graph classes. Cameron et al. \cite{mcm(2021)} proved that, for a finite group $G$, the following are equivalent: $P(G)$ is a threshold graph; $P(G)$ is a split graph; and $G$ belongs to one of the following classes: cyclic groups of prime power order, elementary abelian groups, dihedral $2$-groups, cyclic groups of order $2p$, or dihedral groups of order $2p^{n}$ or $4p$, where $p$ is an odd prime. More recently, Mandal et al. \cite{mm(2023)} classified all the finite groups whose power graphs are $\{P_{5},\overline{P_{5}}\}$-free or $\{P_{2} \cup P_{3}, \overline{P_{2} \cup P_{3}}\}$-free, and further characterised those groups whose power graphs are chain graphs or diamond-free graphs. For recent developments on the power graph over finite groups, we refer the reader to  \cite{kp(2026),stpa(2026), ss(2026), ssk(2026)}. Motivated by these developments, we characterise all the finite groups whose power graphs are friendship graphs, firefly-type graphs, or torch graphs.

We proceed as follows. In Section \ref{s2}, we present the necessary definitions and preliminary results. Section \ref{s3} we classify all finite groups whose power graphs are friendship graphs, firefly-type graphs, or torch graphs. In Section \ref{s4}, we determine the generalised distance spectra of these graph families and derive the corresponding distance and distance signless Laplacian spectra as special cases by considering suitable values of $\alpha$. Finally, we conclude the paper by discussing the main results and suggesting several directions for future research.

\section{Preliminaries}\label{s2}
Let $\Gamma$ be a connected graph with the vertex set $V(\Gamma)$. Two vertices $u, v \in V(\Gamma)$ are said to be adjacent if there is an edge between them, denoted by $u \sim v$. For a vertex $u \in V(\Gamma)$, the neighbourhood of $u$, denoted by $N(u)$, is defined by $N(u)=\{v \ ; \  u \sim v\}$. A clique of a graph $\Gamma$ is the complete subgraph, denoted by $C$. A clique $C$ is the maximal if there is no vertex $v \in V(\Gamma) \setminus V(C)$ such that $V(C) \cup \{v\}$ is complete graph. A maximum clique is a clique of largest size among all cliques of $\Gamma$. Equivalently, the subgraph induced by $C$ is a complete graph. The distance between two vertices $u,v \in V(\Gamma)$, denoted by $d(u,v)$, is the length of the shortest path connecting them. Unless stated otherwise, all graph-theoretic notation and terminology used in this paper are standard and follow \cite{bh(2012), Ssk(2026)}. The distance matrix of $\Gamma$, denoted by $D(\Gamma)$, is defined as $D(\Gamma) = (d(u,v))_{u,v \in V(\Gamma)}$, where each entry represents the distance between the corresponding vertices. The transmission of a vertex $v_i$, denoted as $Tr_\Gamma(v_i)$, is the sum of distances from $v_i$ to all other vertices, that is, $Tr_\Gamma(v_i) = \sum_{u \in V(\Gamma)} d(u,v_i),$  for simplicity, we denote $Tr_\Gamma(v_i)$ as $Tr_i$. The collection $(Tr_1, Tr_2, \ldots, Tr_n)$ is known as the transmission degree sequence of the graph $\Gamma$, where $n$ is the cardinality of the set $V(\Gamma)$. Aouchiche and Hansen \cite{ah(2013), ah(2014)} introduced the distance Laplacian matrix and the distance signless Laplacian matrix of a connected graph $\Gamma$, defined by $D^L(\Gamma)=Tr(\Gamma)-D(\Gamma) \text{and} D^Q(\Gamma)=Tr(\Gamma)+D(\Gamma),$ respectively, where $Tr(\Gamma)=\operatorname{diag}(Tr_1,Tr_2,\ldots,Tr_n)$ is the diagonal matrix of vertex transmissions. Later, Cui et al. \cite{cht(2019)} introduced the generalised distance matrix of a graph $\Gamma$ as a convex combination of $Tr(\Gamma)$ and $D(\Gamma)$ 
$$D_\alpha(\Gamma) = \alpha\,Tr(\Gamma) + (1-\alpha)\,D(\Gamma), \ 0 \le \alpha \le 1.$$ 
This definition unifies several well-known matrices as special cases, such as $D_0(\Gamma) = D(\Gamma), D_1(\Gamma) = Tr(\Gamma),  2D_{\frac{1}{2}}(\Gamma) = D^Q(\Gamma).$ Moreover, the difference between two generalised distance matrices $D_\alpha(\Gamma)$ and $D_\beta(\Gamma)$ is a scalar multiple of the distance Laplacian matrix given as $D_\alpha(\Gamma) - D_\beta(\Gamma) = (\alpha - \beta)\, D^L(\Gamma).$ Therefore, many spectral properties derived for $D_\alpha(\Gamma)$ remain valid for $D(\Gamma)$, $Tr(\Gamma)$, $D^Q(\Gamma)$, and $D^L(\Gamma)$, typically with comparable proof techniques. Thus, $D_\alpha(\Gamma)$ provides a unified framework connecting the spectral theories of the distance and distance signless Laplacian matrices. Since $D_\alpha(\Gamma)$ is a real symmetric matrix, all of its eigenvalues are real. They can be arranged in non-increasing order as $\lambda_1 \ge \lambda_2 \ge \cdots \ge \lambda_n.$ For recent advances on the $D_\alpha$ matrix and its spectral properties, we refer the reader to \cite{abg(2020), abp(2022), abgd(2024), lxs(2021)} and the references therein. 

Let $N$ be an $n \times n$ matrix, and let its rows and columns be grouped according to a partition $\mathcal{P} = \{\mathcal{P}_1, \mathcal{P}_2, \ldots, \mathcal{P}_m\},$ where each $\mathcal{P}_i \subseteq X  = \{1, 2, \ldots, n\}$ and the union of all $\mathcal{P}_i$ gives the full index set $X$. Then the matrix $N$ can be written in block form as:

$$N =
\begin{pmatrix}
N_{1,1} & N_{1,2} & \cdots & N_{1,m} \\
N_{2,1} & N_{2,1} & \cdots & N_{2,m} \\
\vdots & \vdots & \ddots & \vdots \\
N_{m,1} & N_{m,2} & \cdots & N_{m,m}
\end{pmatrix},$$
where each block $N_{i,j}$ corresponds to the submatrix formed by selecting the rows indexed by $\mathcal{P}_i$ and the columns indexed by $\mathcal{P}_j$, for $1 \leq i, j \leq m$. The quotient matrix associated with this block structure is the $m \times m$ matrix $\mathcal{Q} = (\mathcal{Q}_{i,j})$, where each entry $\mathcal{Q}_{i,j}$  is defined as the average row sum of the block $N_{i,j},$ (see \cite[Section 2.3]{bh(2012)}). The partition $\mathcal{P}$ is equitable if every block $N_{i,j}$ has constant row sums (and hence constant column sums, since $N$ is square). In this case, the quotient matrix $\mathcal{Q}$ is called an equitable quotient matrix.

The quotient matrix $\mathcal{Q}$ can also be characterised in terms of an equitable partition of the vertex set of a graph $\Gamma$. A vertex partition $\{V_1, V_2, \ldots, V_m\}$ of the vertex set $V(\Gamma)$ is called equitable if, for each index $i$, and for every pair of vertices $u, v \in V_i$,  $|N(u) \cap V_j| = |N(v) \cap V_j|=q_{ij} \ \text{for all} \ j$. It is easy to see that the quotient matrix $\mathcal{Q}=(q_{ij})$. In general, the eigenvalues of the quotient matrix $\mathcal{Q}$ interlace those of the matrix $M$. In this case, we have the following useful result.
\begin{prop} {\bf (Brouwer and Haemers \cite{bh(2012)})}
If the partition $P$ of rows (or columns) of the matrix $N$ is equitable, then every eigenvalue of the quotient matrix $\mathcal{Q}$ is also an eigenvalue of $N$.
\end{prop}
	
Let $\Gamma$ be a graph whose corresponding matrix is structured in the following block form
	
\begin{equation}\label{1}
M=\begin{pmatrix}
A & A_1 & A_1 & \cdots & A_1 \\
A^t_1 & B & B_1 & \cdots & B_1\\
A^t_1 & B_1 & B & \cdots & B_1\\
\vdots & \vdots & \vdots & \ddots & \vdots\\
A^t_1 & B_1 & B_1  & \cdots &  B
\end{pmatrix}.\end{equation}

Here, the blocks $A$, $A_1$, $B$, and $B_1$ are matrices of orders $t \times t$, $t \times s$, $s \times s$, and $s \times s$, respectively. Then the order of the matrix $M$ is $n = t + m_1 s$, where $m_1$ represents the number of copies of the block $B$. In this setup, the spectrum of the full matrix can be determined by analysing the spectra of the smaller component matrices, see \cite{ft(2016)}. Let $\sigma^k(M)$ denote the multiset consisting of $k$ copies of the spectrum of the matrix $M$, and $\sigma(M)$ represent the set of eigenvalues of $M$. Then the following result describes the relationship between the spectrum of the original matrix and the spectra of its component matrices.

\begin{prop}{\bf (Fritscher and Trevisan \cite{ft(2016)})}\label{p2.2}
If $M$ is the matrix given in (1), with $m_1 (\geq 1$) copies of the block $B$. Then 
\begin{enumerate}
\item $\sigma(B-B_1) \subseteq \sigma(M)$ with multiplicity $m_1 - 1$.
\item $\sigma(M) \setminus \sigma^{(m_1-1)}(B-B_1)=\sigma(M_1)$ is the set of the remaining $t+s$ eigenvalues of $M$, where $M_1=\begin{pmatrix}
A & \sqrt{m_1}A_1\\
\sqrt{m_1}A^t_1 & B+(m_1-1)B_1
\end{pmatrix}$.
\end{enumerate}
\end{prop}

Let $\Gamma_1$ and $\Gamma_2$ be two vertex-disjoint graphs. The \emph{join} of $\Gamma_1$ and $\Gamma_2$, denoted by $\Gamma_1 \vee \Gamma_2$, is the graph obtained by taking the union of $\Gamma_1$ and $\Gamma_2$ and adding an edge between every vertex of $\Gamma_1$ and every vertex of $\Gamma_2$. A fan graph $f_{n}$ is a graph formed by joining a single vertex (called the centre) to all vertices of a path graph $P_{n}$. Equivalently, $f_n=K_1 \vee P_n,$ see Figure 1. The triangular book graph, denoted by $B_n$, is the graph consisting of $n$ triangles sharing a common edge. Equivalently, $B_n=K_2 \vee nK_1.$ Now recall the following graph families that play a central role in this paper.

\begin{deff}\label{d1}\cite{ss(2023)}
A friendship graph, denoted by $F_n$, is formed by joining a single vertex $K_1$ to $n$ disjoint copies of $K_2$, see Figure $2$. Thus the graph $F_n$ is $K_1 \vee (nK_2).$ 
\end{deff}

\begin{deff}\label{d2}\cite{bgp(2023)}
The firefly-type graph, denoted by $F_{p,n-p}$, is obtained from the friendship graph $F_n$ by replacing $p$ ($\geq 1$) copies of its $K_2$ with $K_1$. Consequently, $F_{p,n-p}$ consists of $p$ copies of $K_1$ and $n-p$ copies $K_2$. Thus the graph $F_{p,n-p}$ is $K_1 \vee \left(pK_1\cup (n-p)K_2\right)$, see Figure $3$. 
\end{deff}

\begin{deff}\label{d3}\cite{hk(2022)}
Let $\Gamma_1 = \{v_{n-1}\}$ and $B_n$ the triangular book graph with the common edge $v_{n+3}v_{n+1}$. Let $f_3 =\{v_n\} \vee P_3$, where $P_3 = v_1\,v_{n+2}\,v_{n+4}$. Then a torch graph $O_n$ is the graph with the vertex set  $V(O_n) = V(\Gamma_1)\cup V(B_n) \cup V(f_3)$ and the edge set $E(O_n)  = \{E(\Gamma_1)\cup E(B_n) \cup E(f_3)\} \cup \{v_nv_{n-1},v_{n+3}v_1,v_{n+1}v_1\}$, see Figure $4$. 
\end{deff}

\begin{figure}[H]
\centering
\begin{minipage}[t]{0.48\textwidth}
\centering
\begin{tikzpicture}[thick, scale=.8] 
\setlength{\unitlength}{1mm}
\filldraw[color=black!60, fill=black!5, very thick](2,0);

\draw (0,-.5) node { \scriptsize $\textcolor{black}{\bullet}$};
\draw (0,2.5) node { \scriptsize $\textcolor{black}{\bullet}$};
\draw (-1.5,2.5) node { \scriptsize $\textcolor{black}{\bullet}$};
\draw (-3,2.5) node { \scriptsize $\textcolor{black}{\bullet}$};
\draw (3.5,2.5) node { \scriptsize $\textcolor{black}{\bullet}$};
\draw (0,-.9) node {\scriptsize $\textcolor
	{black}{{v}}$};
\draw (0,3) node {\scriptsize $\textcolor
	{black}{v_{3}}$};    
\draw (-1.5,3) node {\scriptsize $\textcolor
	{black}{v_{2}}$};
\draw (-3,3) node {\scriptsize $\textcolor
	{black}{v_{1}}$};    
\draw (3.5,3) node {\scriptsize $\textcolor
	{black}{v_{p}}$};
\draw[blue] (0,-.5)--(0,2.5);
\draw[blue] (0,-.5)--(-1.5,2.5);
\draw [blue] (0,-.5)--(-1.5,2.5);
\draw[blue] (0,-.5)--(-3,2.5);
\draw [blue] (0,-.5)--(3.5,2.5);
\draw [blue] (0,2.5)--(-1.5,2.5);
\draw[blue] (-3,2.5)--(-1.5,2.5);
\draw [dashed][black] (.3,2.5)--(3.2,2.5);
\draw (-.3, -2) node {\small \textbf{Figure 1}: \textit{Fan graph} $f_{n}$};
\end{tikzpicture} \label{fig2}
\end{minipage}
\hfill
\centering
\begin{minipage}[t]{0.48\textwidth}
\centering
\begin{tikzpicture}[thick, scale=.5] 
\setlength{\unitlength}{1mm}
\filldraw[color=black!60, fill=black!5, very thick](2,0);
\draw (0:0) node {\scriptsize $\textcolor{black}{\bullet}$};
\draw (90:4) node {\scriptsize $\textcolor{black}{\bullet}$};
\draw (62.31:4) node {\scriptsize $\textcolor{black}{\bullet}$};
\draw (34.62:4) node {\scriptsize $\textcolor{black}{\bullet}$};
\draw (6.92:4) node {\scriptsize $\textcolor{black}{\bullet}$};
\draw (-20.77:4) node {\scriptsize $\textcolor{black}{\bullet}$};
\draw (-48.46:4) node {\scriptsize $\textcolor{black}{\bullet}$};
\draw (-76.15:4) node {\scriptsize $\textcolor{black}{\bullet}$};
\draw (-103.85:4) node {\scriptsize $\textcolor{black}{\bullet}$};
\draw (-131.54:4) node {\scriptsize $\textcolor{black}{\bullet}$};
\draw (-159.23:4) node {\scriptsize $\textcolor{black}{\bullet}$};
\draw (-214.62:4) node {\scriptsize $\textcolor{black}{\bullet}$};
\draw (-242.31:4) node {\scriptsize $\textcolor{black}{\bullet}$};
 \draw (-1.5:-.65) node {\scriptsize $\textcolor
 	{black}{{v}}$};
 \draw (90:4.4) node {\scriptsize $\textcolor
 	{black}{v_{1}}$};    
 \draw (62.31:4.4) node {\scriptsize $\textcolor
 	{black}{v_{2}}$};
 \draw (34.62:4.4) node {\scriptsize $\textcolor
 	{black}{v_{3}}$};    
 \draw (6.92:4.4) node {\scriptsize $\textcolor
 	{black}{v_{4}}$};  
 \draw (-20.77:4.4) node {\scriptsize $\textcolor
 	{black}{v_{5}}$};    
 \draw (-48.46:4.4) node {\scriptsize $\textcolor
 	{black}{v_{6}}$}; 
 \draw (-76.15:4.4) node {\scriptsize $\textcolor
 	{black}{v_{7}}$};    
 \draw (-103.85:4.4) node {\scriptsize $\textcolor
 	{black}{v_{8}}$};
 \draw (-131.54:4.4) node {\scriptsize $\textcolor
 	{black}{v_{9}}$};    
 \draw (-162.23:4.6) node {\scriptsize $\textcolor
 	{black}{v_{10}}$};  
 \draw (-214.62:4.7) node {\scriptsize $\textcolor
 	{black}{v_{n-1}}$};    
 \draw (-242.31:4.4) node {\scriptsize $\textcolor
 	{black}{v_{n}}$}; 
\draw[blue] (0:0)--(90:4);
\draw[blue] (0:0)--(62.31:4);
\draw[blue] (90:4)--(62.31:4);
\draw[blue] (0:0)--(34.62:4);
\draw[blue] (0:0)--(6.92:4);
\draw[blue] (34.62:4)--(6.92:4);
\draw[blue] (0:0)--(-20.77:4);
\draw[blue] (0:0)--(-48.46:4);
\draw[blue] (-20.77:4)--(-48.46:4);
\draw[blue] (0:0)--(-76.15:4);
\draw[blue] (0:0)--(-103.85:4);
\draw[blue] (-76.15:4)--(-103.85:4);
\draw[blue] (0:0)--(-131.54:4);
\draw[blue] (0:0)--(-159.23:4);
\draw[blue] (-131.54:4)--(-159.23:4);
\draw[blue] (0:0)--(-214.62:4);
\draw[blue] (0:0)--(-242.31:4);
\draw[blue] (-214.62:4)--(-242.31:4);
\draw[black,dashed] (-210:3.9)--(-162:3.9);
\draw (-.5, -5.5) node {\small \textbf{Figure 2}: \textit{Friendship Graph} $F_n$};
\end{tikzpicture} \label{fig1}
\end{minipage}

\hfill
\centering
\begin{minipage}[t]{0.48\textwidth}
\centering
\begin{tikzpicture}[thick, scale=.5] 
\setlength{\unitlength}{1mm}
\filldraw[color=black!60, fill=black!5, very thick](2,0);
\draw (0:0) node {\scriptsize $\textcolor{black}{\bullet}$};
\draw (90:4) node {\scriptsize $\textcolor{black}{\bullet}$};
\draw (62.31:4) node {\scriptsize $\textcolor{black}{\bullet}$};
\draw (34.62:4) node {\scriptsize $\textcolor{black}{\bullet}$};
\draw (6.92:4) node {\scriptsize $\textcolor{black}{\bullet}$};
\draw (-20.77:4) node {\scriptsize $\textcolor{black}{\bullet}$};
\draw (-48.46:4) node {\scriptsize $\textcolor{black}{\bullet}$};
\draw (-103.85:4) node {\scriptsize $\textcolor{black}{\bullet}$};
\draw (-131.54:4) node {\scriptsize $\textcolor{black}{\bullet}$};
\draw (-150.23:4) node {\scriptsize $\textcolor{black}{\bullet}$};
\draw (-173.92:4) node {\scriptsize $\textcolor{black}{\bullet}$};
\draw (-197.62:4) node {\scriptsize $\textcolor{black}{\bullet}$};
\draw (-248.31:4) node {\scriptsize $\textcolor{black}{\bullet}$};
 \draw (-80:.60) node {\scriptsize $\textcolor
 	{black}{{v}}$};
 \draw (90:4.4) node {\scriptsize $\textcolor
 	{black}{v_{1}}$};    
 \draw (62.31:4.4) node {\scriptsize $\textcolor
 	{black}{v_{2}}$};
 \draw (34.62:4.4) node {\scriptsize $\textcolor
 	{black}{v_{3}}$};    
 \draw (6.92:4.4) node {\scriptsize $\textcolor
 	{black}{v_{4}}$};  
 \draw (-20.77:4.4) node {\scriptsize $\textcolor
 	{black}{v_{5}}$};    
 \draw (-48.46:4.4) node {\scriptsize $\textcolor
 	{black}{v_{6}}$}; 
 \draw (-103.85:4.4) node {\scriptsize $\textcolor
 	{black}{v_{2(n-p)-1}}$};
 \draw (-131.54:4.6) node {\scriptsize $\textcolor
 	{black}{v_{2(n-p)}}$}; 
 \draw (-150.23:4.5) node {\scriptsize $\textcolor
 	{black}{w_{1}}$};  
\draw (-173.92:4.5) node {\scriptsize $\textcolor
 	{black}{w_{2}}$};  
 \draw (-197.62:4.5) node {\scriptsize $\textcolor
 	{black}{w_{3}}$};    
 \draw (-248.31:4.5) node {\scriptsize $\textcolor
 	{black}{w_{p}}$}; 
\draw[blue] (0:0)--(90:4);
\draw[blue] (0:0)--(62.31:4);
\draw[blue] (90:4)--(62.31:4);
\draw[blue] (0:0)--(34.62:4);
\draw[blue] (0:0)--(6.92:4);
\draw[blue] (34.62:4)--(6.92:4);
\draw[blue] (0:0)--(-20.77:4);
\draw[blue] (0:0)--(-48.46:4);
\draw[blue] (-20.77:4)--(-48.46:4);
\draw[blue] (0:0)--(-103.85:4);
\draw[blue] (0:0)--(-131.54:4);
\draw[blue] (-103.85:4)--(-131.54:4);
\draw [dashed][black] (-52.46:3.9)--(-100:3.9);
\draw[blue] (0:0)--(-150.23:4);
\draw[blue] (0:0)--(-173.92:4);
\draw[blue] (0:0)--(-197.62:4);
\draw[blue] (0:0)--(-248.31:4);
\draw [dashed][black] (-203.62:3.9)--(-245:3.9);
\draw (-.5, -5.5) node {\small \textbf{Figure 3}: \textit{Firefly-type graph} $F_{p,n-p}$};
\end{tikzpicture} \label{fig2}
\end{minipage}
\hfill
\begin{minipage}[t]{0.48\textwidth}
\centering
\begin{tikzpicture}[thick, scale=.6] 
\setlength{\unitlength}{1mm}
\filldraw[color=black!60, fill=black!5, very thick](2,0);
\draw (-2,2) node { \scriptsize $\textcolor{black}{\bullet}$};
\draw (-.5,2) node { \scriptsize $\textcolor{black}{\bullet}$};
\draw (1,2) node { \scriptsize $\textcolor{black}{\bullet}$};
\draw (-.5,5.5) node { \scriptsize $\textcolor{black}{\bullet}$};
\draw (-.5,6.5) node { \scriptsize $\textcolor{black}{\bullet}$};
\draw (-.5,8) node { \scriptsize $\textcolor{black}{\bullet}$};
\draw (1,4) node { \scriptsize $\textcolor{black}{\bullet}$};
\draw (-2,4) node { \scriptsize $\textcolor{black}{\bullet}$};
\draw (-.6,.3) node { \scriptsize $\textcolor{black}{\bullet}$};
\draw (-.6,-1) node { \scriptsize $\textcolor{black}{\bullet}$};
\draw (-2.6,2) node {\scriptsize $\textcolor
	{black}{v_{n+4}}$};
\draw (-.5,2.4) node {\scriptsize $\textcolor
	{black}{v_{1}}$};
\draw (1.6,2) node {\scriptsize $\textcolor
	{black}{v_{n+2}}$};
\draw (-.95,.3) node {\scriptsize $\textcolor
	{black}{v_{n}}$};
\draw (-.6,-1.2) node {\scriptsize $\textcolor
	{black}{v_{n-1}}$};  
\draw (1.6,4) node {\scriptsize $\textcolor
	{black}{v_{n+1}}$};
\draw (-2.6,4) node {\scriptsize $\textcolor
	{black}{v_{n+3}}$};  
\draw (-.5,5.1) node {\scriptsize $\textcolor
	{black}{v_{2}}$};  
\draw (-.5,6) node {\scriptsize $\textcolor
	{black}{v_{3}}$};
\draw (-.5,8.2) node {\scriptsize $\textcolor
	{black}{v_{n-2}}$};  
\draw[blue] (-2,2)--(-.6,.3);
\draw[blue] (-.6,.3)--(-.6,-1);
\draw[blue] (1,2)--(-.6,.3);
\draw[blue] (-2,2)--(-.5,2);
\draw[blue] (-.6,.3)--(-.6,-1);
\draw[blue] (1,2)--(-.5,2);
\draw[blue] (-.6,.3)--(-.5,2);
\draw[blue] (1,4)--(-.5,2);
\draw[blue] (-2,4)--(-.5,2);
\draw[blue] (-2,4)--(1,4);
\draw[blue] (1,4)--(-.5,5.5);
\draw[blue] (-2,4)--(-.5,5.5);
\draw[blue] (1,4)--(-.5,6.5);
\draw[blue] (-2,4)--(-.5,6.5);
\draw[blue] (1,4)--(-.5,6.5);
\draw[blue] (-2,4)--(-.5,6.5);
\draw[blue] (1,4)--(-.5,8);
\draw[blue] (-2,4)--(-.5,8);
\draw [dashed][black] (-.5,8)--(-.5,6.5);
\draw (-.5, -2.3) node {\small \textbf{Figure 4}: \textit{Torch graph} $O_n$};
\end{tikzpicture} \label{fig3}
\end{minipage}
\end{figure}

\section{Classification of groups through their power graphs}\label{s3}
Let $G$ be a finite group. The exponent of $G$, denoted by $\exp(G)$, is the smallest positive integer $q$ such that $g^{q}=e, \ \text{for every } g\in G$, where $e$ is the identity element of $G$. Equivalently,
$$\exp(G)=\min\{q\in\mathbb{Z}^{+}: g^{q}=e \text{ for all } g\in G\}
=\operatorname{lcm}\{o(g): g\in G\},$$
where $o(g)$ denotes the order of the element $g$. For the basic concepts and results of group theory, we refer the reader to \cite{g(2021)}.
In this section, we classify the groups of exponent $q$ whose power graphs are the friendship, firefly-type, or torch graphs.

\begin{lem}\label{l1}
Let $G$ be a finite group of exponent $6$. Then
\begin{enumerate}
    \item if $G$ has an element of order $6$, then the power graph $P(G)$ is neither  friendship nor  firefly-type, 
    \item if $G$ has no element of order $6$, then $P(G)\cong F_{p,n-p}$.
\end{enumerate}
\end{lem}

\noindent\textbf{Proof.}
Let $\exp(G)=6$. Then the possible orders of elements in $G$ are $1, 2, 3,$ 
and $6$. Now consider the following two exhaustive cases.

\smallskip
\noindent\textbf{Case 1: $G$ contains an element of order $6$.}  

Suppose that $g\in G$ with $o(g)=6$. Then the elements $g$ and $g^5$ in $\langle g\rangle$ are adjacent to all the elements of $\langle g\rangle$ in the power graph $P(G)$. In particular, the set $N(g^2) \cap N(g^3)$ has more than one elements. Hence $P(G)$ is neither a friendship nor a firefly-type graph.

\noindent\textbf{Case 2: $G$ contains no element of order $6$.} 
Since $\exp(G) = 6$ and $G$ contain no element of order $6$, every non-identity element of $G$ has order $2$ and $3$. 
If we pick an element $g$ of order $3$ in $G$, then $P(\langle g \rangle) \cong K_3 $ and hence for two distinct subgroups $\langle g_1 \rangle$ and $\langle g_2 \rangle$ of order $3$, $P(\langle g_1 \rangle) \cap P(\langle g_2 \rangle) = e$ in $P(G)$. 
Also, taking an element $h$ of order $2$ in $G$, then $P(\langle h \rangle) \cong K_2$, that is,  $h$ is only adjacent to $e$ in $P(G)$. By Definition \ref{d2}, $P(G)$ is firefly-type and hence is not a friendship graph.
\hfill$\blacksquare$

\begin{rem}
Observe that, by Lemma \ref{l1}, if $\exp(G)=6$, then $P(G)$ is not a friendship graph. However, the converse does not hold. For example, consider $G=\mathbb{Z}_p$, where $p\neq3$ is prime, then $P(\mathbb{Z}_p)\cong K_p$,
which is not a friendship graph, while $\exp(G)=p$.
\end{rem}

\begin{thm}\label{t3.1}
Let $G$ be a finite group with exponent $q$. Then $P(G)$ is a friendship graph if and only if $q = 3$.
\end{thm}
\noindent\textbf{Proof.} First we proof the converse part. Suppose $\exp(G) = 3$. Then every non-identity element $g$ of $G$ has order $3$, and hence $P(\langle g \rangle) \cong K_3$ in $P(G)$. Moreover, any two distinct such triangles share only the identity vertex $e$. Therefore
\[
P(G) \cong K_1 \vee mK_2, \ \text{where } m = \frac{|G|-1}{2},
\]
which is the friendship graph.  

Now, suppose $P(G)$ is a friendship graph. Then every maximal clique of $P(G)$ has size exactly $3$. If $\exp(G)$ has a prime divisor $p\geq 5$, then, by Cauchy's theorem, $G$ contains an element $g$ of order $p$. In the case, $P(\langle g \rangle) \cong K_p$, which contradicts that $P(G)$ has a clique of maximum size 3. Hence, every prime divisor of $\exp(G)$ belongs to $\{2, 3\}$, and therefore $\exp(G) = 2^a 3^b$ for some integers $a, b \geq 0$. If $a \geq 2$, then there exists an element $g_1 \in G$ whose order is divisible by $4$. Let $x=g_1^{o(g_1)/4}.$ Then $o(x)=4$, and hence $P(\langle x \rangle) \cong K_4.$ Thus, $P(G)$ contains a clique $K_4$, so it cannot be a friendship graph. Therefore, $a\leq 1$. Similarly, we see that $b \le 1$. So, $\exp(G) \in \{2,3,6\}.$ by Lemma \ref{l1}, $exp(G) \neq 6$.  If $\exp(G)=2$, then $P(G) \cong K_{1, |G|-1},$ and hence not a friendship graph. So, let $\exp(G)=3$. Then by the converse part of this theorem, $P(G)$ is a friendship graph.
\hfill$\blacksquare$

\begin{cor}
Let $G$ be a finite abelian group. Then the power graph $P(G)$ is a friendship graph if and only if $G \cong \mathbb Z_3^{\,n}$ for some positive integer $n$.
\end{cor}
\noindent\textbf{Proof.} The proof follows directly from Theorem \ref{t3.1} and the fundamental theorem for finite abelian groups. \hfill$\blacksquare$

\begin{thm}\label{T2}
Let $G$ be a finite group. Then $P(G)$ is a firefly-type graph if and only if the non-identity elements of $G$ have order $2$ and $3$, only.
\end{thm}

\noindent\textbf{Proof.} Suppose $P(G)$ is a firefly-type graph, that is, $P(G)$ consists of triangles and pendant edges (edges incident to pendant vertices) all sharing the single common vertex $e$. The existence of a pendant vertex $v \neq e$ in $P(G)$ gives $\langle v \rangle = \{e, v\}$, so $o(v)=2$. The existence of a triangle $\{e, u, u^2\}$ in $P(G)$ gives $\langle u \rangle = \{e, u, u^2\}$, so $o(u) = 3$. Furhter, if $G$ contains an element $g$ of $o(r) \geq 4$, then the power graph $P(G)$ contains a clique of size at least $4$, which is not possible. Hence, the set of orders of non-identity elements of $G$ is exactly $\{2, 3\}$. The converse part directly follows from the Case $2$ of Lemma \ref{l1}.
\hfill$\blacksquare$

\begin{cor}\label{c2}
Among all symmetric groups $S_n$, the power graph $P(S_n)$ is a firefly graph if and only if $n = 3$.
\end{cor}

\noindent\textbf{Proof.}
By Theorem \ref{T2}, $P(S_n)$ is a firefly graph if and only if the set of orders of non-identity elements of $S_n$ is exactly $\{2, 3\}$. For $n \neq 3$, $P(S_n)$ is not firefly-type as $S_n$ has an element of order neither $2$ nor $3$. For $n = 3$, $P(S_3)$ is a firefly-type graph, precisely $P(S_3) \cong K_1 \vee (3K_1 \cup K_2)$. \hfill$\blacksquare$

\begin{cor}
Let $A_n$ be the alternating group of degree $n$. Then $P(A_n)$ is a firefly-type graph if and only if $n = 4$.
\end{cor}

\noindent\textbf{Proof.} The proof is similar to that in the Corollary \ref{c2}. For $n \neq 4$, $P(A_n)$ is not firefly-type as $A_n$ has an element of order neither $2$ nor $3$. For $n = 4$, $P(A_4)$ is a firefly-type graph, precisely $P(A_4) \cong K_1 \vee (3K_1 \cup 4K_2)$.
\hfill$\blacksquare$

\begin{rem}
Note that when $n=3$, we have $A_n \cong \mathbb{Z}_3$, and every non-identity element has order $3$. Hence, $P(A_3)\cong K_3$, which is a friendship graph.
\end{rem}

\begin{thm}
There is no group whose power graph is isomorphic to the torch graph $O_n$. 
\end{thm}

\noindent\textbf{Proof.}
Let $O_n$ be the torch graph and $P(G)$ the power graph of a group $G$. Let the number of elements in $O_n$ (see Figure 4) and $G$ be $n+4$. The proof now follows by considering the maximum degree $\Delta$ in both the graphs. Note that $\Delta (O_n) = n-1$ while $\Delta (P(G)) = n+3$. So, $P(G) \ncong O_n$. 
\hfill$\blacksquare$

\section{Generalized distance spectrum of special classes of graphs}\label{s4}

In the preceding section, we characterised the finite groups whose power graphs are isomorphic to the friendship graph, the firefly-type graph, and the torch graph. The spectral properties of these graph families have been investigated by several researchers \cite{aj(2014),am(2019)}. In \cite{ss(2023)}, the authors studied the spectral radius of the friendship and weak friendship graphs. In \cite{bgp(2023)},  the authors obtained the $A_\alpha$-spectra of the friendship and firefly-type graphs. Furthermore, \cite{eofl(2019)} derived a partial factorisation of the $A_\alpha$-characteristic polynomial of the firefly graph, yielding several explicit eigenvalues. Motivated by these results and our characterisation of the corresponding power graphs, here, we discuss the generalised distance spectra of these graph families.

\begin{thm}\label{t1}
The generalised distance spectrum of the friendship graph $F_n$ is 
\[
\operatorname{Spec}_{D_\alpha}(F_{n})
=
\left\{
((4n + 1)\alpha - 3)^{\,n-1},
((4n - 1)\alpha - 1)^{\,n},
\lambda_+,\lambda_{-}
\right\}, \, \text{where}
\] 
\[
   \lambda_{\pm}= \frac{(2n + 1)\alpha + 4n - 3 \pm 
    \sqrt{\left((2n + 1)\alpha + 4n - 3\right)^2 
    - 8n(4n\alpha-\alpha-1)}}{2}.
    \]
\end{thm}

\noindent{\textbf{Proof.}} Let the vertex set $V(F_n \cong K_1 \vee nK_2) = \{v, v_{11}, v_{12}, v_{21}, v_{22}, \ldots, v_{n1}, v_{n2}\}$, where $v \sim v_{ij}$ and $v_{i1} \sim v_{i2}$, for $1 \leq i \leq n$ and $1 \leq j \leq 2$. Observed that the transmission of vertex $v$ is $Tr(v) = 2n$ and $Tr(v_{ij}) = 4n - 2$ for $1 \le i \le n$ and $ 1\leq j \leq 2 $. So, the generalised distance matrix of the graph $F_n$ is given by
\[
D_\alpha(F_n) = 
\left(
\begin{array}{c|cc|cc|c|cc}
2n\alpha & 1-\alpha & 1-\alpha & 1-\alpha & 1-\alpha & \cdots & 1-\alpha & 1-\alpha \\ \hline
1-\alpha & (4n-2)\alpha & 1-\alpha & 2(1-\alpha) &  2(1-\alpha) & \cdots &  2(1-\alpha) &  2(1-\alpha) \\
1-\alpha & 1-\alpha & (4n-2)\alpha & 2(1-\alpha) & 2(1-\alpha) & \cdots &  2(1-\alpha) &  2(1-\alpha) \\ \hline
1-\alpha &  2(1-\alpha) &  2(1-\alpha) & (4n-2)\alpha & 1-\alpha & \cdots &  2(1-\alpha) &  2(1-\alpha) \\
1-\alpha &  2(1-\alpha) &  2(1-\alpha) & 1-\alpha & (4n-2)\alpha & \cdots &  2(1-\alpha) &  2(1-\alpha) \\ \hline
\vdots & \vdots & \vdots & \vdots & \vdots & \ddots & \vdots & \vdots \\ \hline
1-\alpha & 2(1-\alpha) &  2(1-\alpha) &  2(1-\alpha) &  2(1-\alpha) & \cdots & (4n-2)\alpha & 1-\alpha \\
1-\alpha &  2(1-\alpha) &   2(1-\alpha) &   2(1-\alpha) &   2(1-\alpha) & \cdots & 1-\alpha & (4n-2)\alpha \\
\end{array}
\right)
.\]
The above matrix admits a block matrix representation as described in (1) of Section \ref{s2}, with the  components $$A=2n\alpha, A_1=\begin{pmatrix}1-\alpha & 1-\alpha\end{pmatrix}, B=\begin{pmatrix} (4n-2)\alpha & 1-\alpha\\
1-\alpha & (4n-2)\alpha\end{pmatrix}, \, \text{and} \, B_1=\begin{pmatrix}
      2(1-\alpha) &   2(1-\alpha) \\
      2(1-\alpha) &   2(1-\alpha)
\end{pmatrix},$$ where $B$ occurs $m_1=n$ times. Then the matrix $B-B_1$ has eigenvalues $(4n+1)\alpha-3$ and $(4n-1)\alpha-1$. Now, define the matrix $$M_1=\begin{pmatrix}
    2n\alpha & \sqrt{n}(1-\alpha) & \sqrt{n}(1-\alpha)\\
    \sqrt{n}(1-\alpha) & 2n\alpha+2n-2 & (2n-1)(1-\alpha)\\
    \sqrt{n}(1-\alpha) & (2n-1)(1-\alpha) & 2n\alpha+2n-2
\end{pmatrix}.$$
The eigenvalues of $M_1$ are $$(4n-1)\alpha-1, \ \lambda_{\pm}= \frac{(2n+1)\alpha+4n-3 \pm \sqrt{\left((2n+1)\alpha+4n-3\right)^2-8n(4n\alpha-\alpha-1)}}{2}.$$

By Proposition~\ref{p2.2}, the $D_\alpha(F_n)$ spectrum consists of eigenvalues $\lambda_{\pm}$, together with $(4n+1)\alpha-3$ and $(4n-1)\alpha-1$ with multiplicities $n-1$ and $n$, respectively. Hence, the result follows.
\hfill$\blacksquare$

By setting $\alpha = 0$, we obtain $D_0(G)=D(G)$. Applying Theorem~\ref{t1} in this special case yields the following result describing the distance spectrum of the friendship graph.

\begin{cor}
The distance spectrum of the friendship graph \(F_n\) is 
\[
\operatorname{Spec}_{D}(F_{n})
=
\left\{
(- 3)^{\,n-1},
(- 1)^{\,n}, 
\lambda_+,\lambda_{-}
\right\}, \, \text{where} \ 
\lambda_{\pm}
=
\frac{4n-3 \pm \sqrt{16n^2-16n+9}}{2}.
\]
\end{cor}

By setting $\alpha = \tfrac{1}{2}$, we obtain $2D_{\frac{1}{2}}(G)=D^{Q}(G)$. Thus, by applying Theorem \ref{t1}, we obtain the following result for the distance signless Laplacian spectrum of the friendship graph.

\begin{cor}
The distance signless Laplacian spectrum of the friendship graph \(F_n\) is 
$$\operatorname{Spec}_{D}(F_{n})
=
\left\{
(4n-5)^{\,n-1},
(4n-3)^{\,n}, 
\lambda_+,\lambda_{-}
\right\}, \, \text{where} \, \lambda_{\pm}
=
\frac{10n-5 \pm \sqrt{36n^2-52n+25}}{2}.$$

\end{cor}

\noindent The next theorem extends this analysis to another family of graphs. In particular, we determine the generalised distance spectrum of the firefly-type graph, which depends on the parameter $\alpha$.

\begin{thm}\label{t2}
The generalised distance spectrum of the friendship graph $F_{p,n-p}$ is 

\noindent $\operatorname{Spec}_{D_\alpha}(F_{p,n-p})
=\Big\{
((4n-2p+1)\alpha-2)^{p-1},
((4n-2p-1)\alpha-1)^{n-p}, ((4n-2p+1)\alpha-3)^{n-p-1},
\lambda_1,\lambda_2,\lambda_3
\Big\},$ $\text{ where } \lambda_1, \lambda_2,  \text{and} \ \lambda_3$ are the roots of the cubic polynomial $p(x)=-x^3 + \big((6n - 3p + 2)\alpha + 4n - 2p - 5\big)x^2 + \big((-8n^2 - 2p^2 + 8np + 3p - 6n - 1)\alpha^2 + (-24n^2 - 6p^2 + 24np - 9p + 14n + 5)\alpha + 10n - 3p - 6\big)x + \big((32n^3 - 4p^3 - 2p^2 - 48n^2p + 24np^2 + 4np + 2p - 2n)\alpha^2 - (24n^2 + 4p^2 - 20np + 3p - 2n)\alpha + 4n - p\big).$

\end{thm}

\noindent{\textbf{Proof.}} Let $F_{p, n-p}$ be the firefly-type graph with the vertex set partitioned as $V(F_{p, n-p}) = \{v\} \cup \{w_1, w_2, \ldots, w_p\} \cup \{v_1, v_2, \ldots, v_{2(n-p)}\}$, where, $v$ is the  central vertex. Vertices $w_1, \ldots, w_p$ are the pendant vertices connected directly to $v$. The remaining $2(n-p)$ vertices form $(n-p)$ disjoint triangles, each sharing the central vertex $v$. The transmission degrees of each vertex are
\[Tr(v) = 2n - p, \, Tr(w_i) = 4n - 2p - 1, \, \text{and} \, Tr(v_j) = 4n - 2p - 2, \, \text{where} \, 1 \le j \le 2(n-p),  1 \le i \le p. \]
Then the generalised distance matrix of the firefly-type graph $F_{p,n-p}$ is given by
    \[
D_\alpha(F_{p,n-p}) = 
\left(
\begin{array}{c|cccc|ccccccc}
g\alpha & \beta & \beta & \cdots & \beta & \beta & \beta & \beta & \beta & \cdots & \beta  & \beta \\ \hline
\beta & g_1\alpha & 2\beta & \cdots & 2\beta & 2\beta & 2\beta & 2\beta & 2\beta & \cdots & 2\beta & 2\beta \\
\beta & 2\beta & g_1\alpha & \cdots & 2\beta & 2\beta & 2\beta & 2\beta & 2\beta & \cdots & 2\beta & 2\beta \\
\vdots & \vdots & \vdots & \ddots & \vdots & \vdots & \vdots & \vdots & \vdots & \ddots & \vdots & \vdots \\
\beta & 2\beta & 2\beta & \cdots & g_1 \alpha & 2\beta & 2\beta & 2\beta & 2\beta & \ddots & 2\beta & 2\beta \\ \hline
\beta & 2\beta & 2\beta & \cdots & 2\beta & g_2\alpha & \beta & 2\beta & 2\beta & \cdots & 2\beta & 2\beta\\
\beta & 2\beta & 2\beta & \cdots & 2\beta & \beta & g_2\alpha & 2\beta & 2\beta & \cdots & 2\beta & 2\beta \\
\beta & 2\beta & 2\beta & \cdots & 2\beta & 2\beta & 2\beta & g_2\alpha & \beta & \cdots & 2\beta & 2\beta \\
\beta & 2\beta & 2\beta & \cdots & 2\beta & 2\beta & 2\beta & \beta & g_2\alpha & \cdots & 2\beta & 2\beta \\
\vdots & \vdots & \vdots & \ddots & \vdots & \vdots & \vdots & \vdots & \vdots & \ddots & \vdots & \vdots\\
\beta & 2\beta & 2\beta & \cdots & 2\beta & 2\beta & 2\beta & 2\beta & 2\beta & \ddots & g_2 \alpha & \beta\\
\beta & 2\beta & 2\beta & \cdots & 2\beta & 2\beta & 2\beta & 2\beta & 2\beta & \ddots & \beta & g_2 \alpha
\end{array}
\right)
,\] where $\beta=1-\alpha$, $g=2n-p$, $g_1=4n-2p-1$, and $g_2=4n-2p-2$.

Since $D_{\alpha}(F_{p,n-p})$ is a symmetric matrix, its diagonal block matrices are also symmetric. The second and third diagonal blocks have constant row sums, implying that these constants are eigenvalues of the respective blocks corresponding to the eigenvector $\mathbf{1} = (1,1,\ldots,1)^T$.
The remaining eigenvectors are orthogonal to $\mathbf{1}$, meaning that their coordinate sums are zero. To construct such eigenvectors, consider the vectors 
\[
Z_i = e_i - e_{i+1}, \ 2 \le i \le p,
\]
where \( e_i \) denotes the standard basis vector having \( 1 \) in the \( i^{\text{th}} \) position and zeros elsewhere. 
Then,
\[
D_{\alpha}(F_{p,n-p})Z_i = \big((4n - 2p + 1)\alpha - 2\big)Z_i.
\]
Hence, $(4n - 2p + 1)\alpha - 2$ is an eigenvalue of $D_{\alpha}(F_{p,n-p})$ with multiplicity $p - 1$.

Next, consider the vectors 
\[
Y_i = e_{p+2i} - e_{p+2i+1}, \ 1 \leq i \leq n-p.
\]
Then,
\[
D_{\alpha}(F_{p,n-p})Y_i = \big((4n - 2p - 1)\alpha - 1\big)Y_i.
\]
Hence, $(4n - 2p - 1)\alpha - 1$ is an eigenvalue of $D_{\alpha}(F_{p,n-p})$ with multiplicity $n - p$.

Now, consider the vectors 
\[
X_i = -e_{p+2i} - e_{p+2i+1} + e_{p+2i+2} + e_{p+2i+3}, \ 1 \leq i \leq n-p-1.
\]
Then,
\[
D_{\alpha}(F_{p,n-p})X_i = \big((4n - 2p + 1)\alpha - 3\big)X_i.
\]
Hence, $(4n - 2p + 1)\alpha - 3$ is an eigenvalue of $D_{\alpha}(F_{p,n-p})$ with multiplicity $n - p - 1$.

The remaining three eigenvalues of $D_{\alpha}(F_{p,n-p})$ can be obtained from the characteristic polynomial $p(x)$ corresponding to the equitable quotient matrix, given in the statement of the theorem

$$\begin{pmatrix}
    (2n-p)\alpha & p(1-\alpha) & 2(n-p)(1-\alpha)\\
    1-\alpha & (4n-4p+1)\alpha+2(p-1) & 4(n-p)(1-\alpha)\\
    1-\alpha & 2p(1-\alpha) & (2p+1)\alpha+4n-4p-3
\end{pmatrix}.$$ 
\hfill$\blacksquare$

By setting $\alpha = 0$, we obtain $D_0(G)=D(G)$. Applying Theorem~\ref{t2} in this special case yields the following result describing the distance spectrum of the firefly-type graph.

\begin{cor}
The distance spectrum of the firefly-type graph is
\[
\operatorname{Spec}_{D}(F_{n,p})
=
\left\{
(-2)^{\,p-1},
(-1)^{\,n-p},
(-3)^{\,n-p-1},
\lambda_1,\lambda_2,\lambda_3
\right\},
\]
where \(\lambda_1,\lambda_2,\lambda_3\) are the roots of polynomial $ p(x)=x^3-(4n-2p-5)x^2-(10n-3p-6)x-(4n-p).$
\end{cor}

\noindent\textbf{Proof.}  
Substituting $\alpha = 0$ in Theorem \ref{t2} yields the eigenvalues $-2$, $-1$, and $-3$ with respective multiplicities $p - 1$, $n - p$, and $n - p - 1$. The remaining three eigenvalues $\lambda_1$, $\lambda_1,$ and $\lambda_3$ are the roots of the characteristic polynomial $p(x)$ of the following equitable quotient matrix is
\[
\begin{pmatrix}
    0 & p & 2(n-p)\\
    1 & 2(p-1) & 4(n-p)\\
    1 & 2p & 4n - 4p - 3
\end{pmatrix}.
\]
\hfill$\blacksquare$

By setting $\alpha = \tfrac{1}{2}$, we obtain $2D_{\frac{1}{2}}(G)=D^{Q}(G)$. Using Theorem \ref{t2}, we derive the following result for the distance signless Laplacian spectrum of the firefly-type graph.

\begin{cor}
The distance signless Laplacian spectrum of the firefly-type graph is
\[
\operatorname{Spec}_{Q_D}(F_{n,p})
=
\left\{
(4n-2p-3)^{\,n-1},
(4n-2p-5)^{\,n-p-1},
\lambda_1,\lambda_2,\lambda_3
\right\}, \, \text{where} \,
\]
 \(\lambda_1,\lambda_2,\lambda_3\) are the roots of $p(x)=x^3-\left(14n-7p-8\right)x^2+\left(56n^2+14p^2-62n+27p-56np+15\right)x-64n^3+8p^3+96n^2+20p^2+96n^p-48np^2-88np+16p.$
\end{cor}

The following result gives the generalised distance spectrum of the torch graph $(O_n)$.

\begin{thm}\label{t3} 
The distance spectrum of the torch graph $(O_n)$ is
\[
\operatorname{Spec}_{D_\alpha}(O_n)
=
\left\{
((2n+11)\alpha-2)^{\,n-4},
(3n+3)\alpha-2, (n+9)\alpha-1,
\lambda_1,\ldots,\lambda_6
\right\},
\]
where \(\lambda_1,\ldots,\lambda_6\) are the eigenvalues of the following equitable quotient matrix 
    $$ M=\begin{pmatrix}
         (2n+1)\alpha & 2\gamma & 2\beta & \beta & 2\beta & 2\beta\\
         2\beta & 17\alpha+2n-8 & 4\beta & 3\beta & 2\beta & 6\beta\\
         2\beta & 4\gamma & (4n+1)\alpha & \beta & 6\beta & 4\beta\\
         \beta & 3\gamma & \beta & (3n-1)\alpha & 4\beta & 2\beta\\
         \beta & \gamma & 3\beta & 2\beta & (n+7)\alpha+1 & 4\beta\\
         \beta & 3\gamma & 2\beta & \beta & 4\beta & (3n-1)\alpha+2
     \end{pmatrix},$$ where $\gamma=(1-\alpha)(n-3)$ and $\beta=1-\alpha$.
\end{thm}

\noindent{\textbf{Proof.}}
Using Definition \ref{d3}, the vertex set of the torch graph is given by $V(O_n)=\{v_i \,; 1 \le i \leq n+4 \}$. The transmission degree of $\operatorname{Tr}(v_1) = 2n + 1$, $\operatorname{Tr}(v_i) = 2n + 9$ for $2 \le i \le n - 2$, $\operatorname{Tr}(v_{n-1}) = 4n + 1$, $\operatorname{Tr}(v_n) = 3n - 1$, $\operatorname{Tr}(v_j) = n + 8$ for $j = \{n+1,n+3\}$, and $\operatorname{Tr}(v_k) = 3n + 1$ for $k = \{n+2,n+4\}$. By assigning an appropriate ordering to the vertices, the generalised distance matrix of the torch graph $(O_n)$ is given as
$$D_{\alpha}(O_n)=\left(\begin{array}{c|c|c|c|c|c|c|c}
(2n+1)\alpha & 2\beta J & 2\beta & \beta & \beta & \beta & \beta & \beta\\  \hline
2\beta J^t & A_{n-3} & 4\beta J & 3\beta J & \beta J & 3\beta J & \beta J & 3\beta J\\ \hline
2\beta & 4\beta J^t & (4n+1)\alpha & \beta & 3\beta & 2\beta & 3\beta & 2\beta\\ \hline
\beta &  3\beta J^t & \beta & (3n-1)\alpha & 2\beta & \beta & 2\beta & \beta\\ \hline
\beta & \beta J^t & 3\beta & 2\beta & \gamma & 2\beta & \beta & 2\beta \\ \hline
\beta & 3\beta J^t & 2\beta & \beta & 2\beta & \gamma_1 & 2\beta & 2\beta\\ \hline
\beta & \beta J^t & 3\beta & 2\beta & \beta & 2\beta & \gamma & 2\beta\\ \hline
\beta & 3\beta J^t & 2\beta & \beta & 2\beta & 2\beta & 2\beta & \gamma_1
\end{array}\right),$$ where $\gamma=(n+8)\alpha$, $\gamma_1=(3n+1)\alpha$,  $\beta=1-\alpha$, $A_{n-3}=(2n+9)\alpha I_{n-3}+2(1-\alpha)(J-I)_{n-3}$, and $J$ denotes the matrix whose all entries are equal to $1$.

Since all row sums of $A_{n-3}$ are equal to $17\alpha+2n-8$, this quantity is an eigenvalue of $A_{n-3}$ corresponding to the eigenvector $\mathbf{1}=(1,1,\ldots,1)^t$. Moreover, every eigenvector associated with any other eigenvalue of $A_{n-3}$ is orthogonal to $\mathbf{1}$. Let $X_i =(x_{i_1},x_{i_2},\ldots,x_{i_{n-3}})$, for $2 \le i \le n-3$ such that $\mathbf{1}^t X_i=0$. This shows the sum of the component of $X_i$ is zero. These eigenvectors can be naturally extended to eigenvectors of $D_\alpha(O_n)$. In particular, for the vector $Y_i=e_{i+1}-e_{i+2}$, for $1 \leq i \leq n-4$ such that 
$$D_\alpha(O_n)Y_i=[(2n+11)\alpha-2]Y_i, \ \text{for} \ 1 \leq i \leq n-4.$$ Thus we get eigenvalue $(2n+11)\alpha-2$ with multiplicity $n-4$. The rest of the eigenvalues can be found with the help of following equitable quotient matrix 
$$M_1=\begin{pmatrix}
(2n+1)\alpha & 2\gamma & 2\beta & \beta & \beta & \beta & \beta & \beta \\
2\beta & \gamma_1 & 4\beta & 3\beta & \beta & 3\beta & \beta & 3\beta \\
2\beta & 4\gamma & (4n+1)\alpha & \beta & 3\beta & 2\beta & 3\beta & 2\beta\\
\beta & 3\gamma & \beta & (3n-1)\alpha & 2\beta & \beta & 2\beta & \beta\\
\beta & \gamma & 3\beta & 2\beta & \gamma &  2\beta & \beta & 2\beta \\
\beta & 3\gamma & 2\beta & \beta & 2\beta & \gamma_1 & 2\beta & 2\beta \\
\beta & \gamma & 3\beta & 2\beta & \beta & 2\beta & \gamma & 2\beta \\
\beta & 3\gamma & 2\beta & \beta & 2\beta & 2\beta & 2\beta & \gamma_1
\end{pmatrix},$$ where $\gamma_2=(1 - \alpha)(n-3)$, and $\gamma_3 = 17\alpha+2n-8$.
By defining the vector $X_1=(0,0,0,0,1,0,-1,0)^t$ and $X_2=(0,0,0,0,0,1,0,-1)^t$, we observe that
$$M_1X_1=[(n+9)\alpha - 1]X_1 \ \text{and} \ M_1X_2=[(3n+3)\alpha - 2]X_2.$$
which gives two additional eigenvalues of $D_\alpha(O_n)$. The remaining eigenvalues can be obtained from the equitable quotient matrix $M$ defined earlier in the theorem. This completes the proof.
\hfill$\blacksquare$

For different values of $\alpha$ in Theorem \ref{t3}, we get different spectra of the torch graph. Some of them are given as follows.

\begin{cor}
The distance spectrum of the torch graph is
\[
\operatorname{Spec}_{D}(O_n)
=
\left\{
(-2)^{\,n-3},
-1,
\lambda_1,\ldots,\lambda_6
\right\},
\]
where \(\lambda_1,\ldots,\lambda_6\) are the roots of
$p(x)=x^6-(2n-5)x^5+(43n-63)x^4-(123n-103)x^3-(76n+174)x^2+(10n-276)x+12n-88.$
\end{cor}

\noindent{\textbf{Proof.}}
For $\alpha=0$, we have $D_0(G)=D(G)$. By Theorem \ref{t3}, the distance spectrum of the torch graph consists of the eigenvalues $-2$ and $-1$ with multiplicities $n-3$ and $1$. The remaining $6$ eigenvalues are the roots of the characteristic polynomial $p(x)$ of the following equitable quotient matrix

\[
M = \begin{pmatrix}
0 & 2n-6 & 2 & 1 & 2 & 2 \\
2 & 2n - 8 & 4 & 3 & 2 & 6 \\
2 & 4n-12 & 0 & 1 & 6 & 4 \\
1 & 3n-9 & 1 & 0 & 4 & 2 \\
1 & n-3 & 3 & 2 & 1 & 4 \\
1 & 3n-9 & 2 & 1 & 4 & 2
\end{pmatrix}.
\]
\hfill$\blacksquare$

\begin{cor}
The distance signless Laplacian spectrum of the torch graph is
\[
\operatorname{Spec}_{Q_D}(O_n)
=
\left\{
(2n+7)^{\,n-4},
\,3n-1,
\,n+7,
\,\lambda_1,\ldots,\lambda_6
\right\},
\]
where \(\lambda_1,\ldots,\lambda_6\) are the roots of polynomial
$p(x)=x^6 - (17n + 14)x^5 + (117n^2 + 164n + 144)x^4 - (415n^3 + 794n^2 - 1226n - 546)x^3+ (794n^4 + 2070n^3 + 3083n^2 + 6168n - 2153)x^2- (768n^5 + 2974n^4 + 1375n^3 + 21150n^2 - 12299n + 920)x + 288n^6 + 1848n^5 - 2140n^4 + 22252n^3 - 16116n^2 + 988n + 752$. 
\end{cor}

\noindent{\textbf{Proof.}
Setting $\alpha = \tfrac{1}{2}$ gives $2D_{\frac{1}{2}}(G) = D^{Q}(G)$.  
Applying Theorem \ref{t3} to this special case yields the eigenvalues 
$2n+7$ (with multiplicity $n-4$), $3n-1$, and $n+7$. The other six eigenvalues can be determined from the characteristic polynomial $p(x)$ of the following quotient matrix 
\[
M=\begin{pmatrix}
2n+1 & 2n-6 & 2 & 1 & 2 & 2\\
2 & 4n+1 & 4 & 3 & 2 & 6\\
2 & 4n-12 & 4n+1 & 1 & 6 & 4\\
1 & 3n-9 & 1 & 3n-1 & 4 & 2\\
1 & n-3 & 3 & 2 & n+9 & 4\\
1 & 3n-9 & 2 & 1 & 2 & 3n+3
\end{pmatrix}.
\]

\hfill$\blacksquare$

\section{Conclusion}
In this paper, we characterised finite groups of exponent $q$ whose power graphs are friendship graphs, firefly-type graphs, or torch graphs. In particular, we proved that the power graph of a finite group of exponent $q$ is a friendship graph if and only if $q=3$, and that a finite abelian group has a friendship power graph precisely when it is isomorphic to $\mathbb{Z}_3^{n}$. We also showed that, among the symmetric and alternating groups, only $S_3$ and $A_4$ have firefly-type power graphs, and that no finite group has a power graph isomorphic to a torch graph. Furthermore, we determined the $D_{\alpha}$-spectra of friendship, firefly-type, and torch graphs, from which the corresponding distance and distance signless Laplacian spectra follow as special cases.
The results presented here contribute to the understanding of the interplay between finite groups, power graphs, and spectral graph theory. Future work may focus on characterising algebraic structures whose associated graphs belong to other graph classes, and on investigating generalised distance spectral properties.

\bigskip
\noindent{\bf Acknowledgements :} The first author is grateful to DST INSPIRE (Award No. 03/2022/002991), New Delhi, India, for financial support.
\\ The Second author is grateful to UGC- NET, Govt of India, New Delhi, for financial support.

\bigskip
\noindent{\bf Conflict of Interest:} The authors declare that there is no conflict of interest regarding the publication of this work.

\bigskip
\noindent{\bf Data Availability Statement:} No datasets were generated, analysed, or utilised during the preparation of this study.

\end{document}